\documentclass[12pt,reqno,twoside]{amsart}

\usepackage{amsmath}

\usepackage{amssymb}

\usepackage{epsfig}

\usepackage{color}

\numberwithin{equation}{section}

\newif\ifdraft\drafttrue

\draftfalse

\long\def\comlior#1{\ifdraft{\marginpar{\sn
#1 \ LF}}\else\ignorespaces\fi}

\newcommand{\br}{{\mathbb{R}}}

\newcommand{\ct}{\mathcal T}

\font\sb = cmbx8 scaled \magstep0

\font\sn = cmssi8 scaled \magstep0

\long\def\comdima#1{\ifdraft{\marginpar{\sb
#1 \ DK}}\else\ignorespaces\fi}

\newcommand\ba{badly approximable}

\newcommand\da{Diophantine approximation}

\newcommand\ssm{\smallsetminus}

\newcommand\name[1]{\label{#1}{\ifdraft{\sn [#1]}\else\ignorespaces\fi}}

\newcommand\eq[2]{{\ifdraft{\ \tt [#1]}\else\ignorespaces\fi}\begin{equation}\label{eq:#1}{#2}\end{equation}}

\newcommand {\equ}[1]     {\eqref{eq:#1}}

\newcommand{\goth}[1]{{\mathfrak{#1}}}

\newcommand\z{\goth z}

\newcommand{\Q}{{\mathbb {Q}}}

\newcommand{\R}{{\mathbb{R}}}

\newcommand{\T}{{\mathbb{T}}}

\newcommand{\Z}{{\mathbb{Z}}}

\newcommand{\N}{{\mathbb{N}}}

\newcommand{\ubf}{\mathbf{u}}
\newcommand{\vbf}{\mathbf{v}}

\newcommand{\w}{\mathbf{w}}

\newcommand{\GL}{\operatorname{GL}}

\newcommand{\supp}{\operatorname{supp}}

\newcommand {\ignore}[1]  {}

\newcommand{\df}{{\, \stackrel{\mathrm{def}}{=}\, }}

\newcommand{\x}{{\mathbf{x}}}

\newcommand{\ve}{{\bf e}}

\newcommand{\vre}{\varepsilon}

\newcommand\hd{Hausdorff dimension}

\newcommand\nz{\smallsetminus \{0\}}

\newcommand{\rn}{\mathbb{R}^n}

\newcommand{\rmb}{\mathbb{R}^m}
\newcommand{\zm}{\mathbb{Z}^m}
\newcommand{\zn}{\mathbb{Z}^n}

\newcommand{\q}{\mathbf{q}}

\newcommand{\y}{\mathbf{y}}

\newcommand{\cy}{{\mathcal Y}}

\newcommand{\ca}{\mathcal Z}
\newcommand{\cl}{\mathcal L}
\newcommand{\cm}{\mathcal M}

\newcommand{\M}{\operatorname{M}}

\newtheorem{thm}{Theorem}[section]
\newtheorem{lem}[thm]{Lemma}
\newtheorem{prop}[thm]{Proposition}
\newtheorem{cor}[thm]{Corollary}

\newtheorem{remark}[thm]{Remark}
\newtheorem{defn}[thm]{Definition}

\begin{document}

\title[Schmidt's game, fractals, and orbits of toral endomorphisms]{Schmidt's game, fractals, and \\ orbits of toral endomorphisms}
\author[Broderick, Fishman, Kleinbock]{Ryan Broderick, Lior Fishman and Dmitry Kleinbock\\
 }

\address{Department of Mathematics, Brandeis University, Waltham MA 02454-9110} 
\email{ {\tt ryanb@brandeis.edu}, {\tt lfishman@brandeis.edu}, {\tt kleinboc@brandeis.edu}}

\date{May  2010}
\begin{abstract}

Given an integer matrix  $M\in \GL_n(\R)$ and a point $y \in \mathbb{R}^n/\mathbb{Z}^n$, consider the set $$
\tilde E(M,y) \df \left\{\x\in \mathbb{R}^n: y\notin\overline{ \{M^k \x \bmod \Z^n: k\in\N\}}\right\}\,.$$
S.G.\ Dani showed in 1988 that whenever $M$ is semisimple
 and $y \in \mathbb{Q}^n/\mathbb{Z}^n$,  the set $
\tilde E(M,y)$ 
has full Hausdorff dimension.
In this paper we strengthen this result, extending it to arbitrary $M\in \GL_n(\R) \cap \M_{n\times n}(\Z)$ and $y \in \mathbb{R}^n/\mathbb{Z}^n$, and in fact replacing the sequence of powers of $M$ by any 
lacunary sequence of (not necessarily integer) $m\times n$ matrices. Furthermore, we show that sets of the form $
\tilde E(M,y)$ and their generalizations always intersect with  `sufficiently regular' fractal subsets of $ \mathbb{R}^n$.  As an application we 
give an alternative proof of a recent result of \cite{ET}
on badly approximable systems of affine forms. 
\end{abstract}
\maketitle

\section{Introduction} \label{intro} 
Let $\T^n \df \R^n/\Z^n$ be the $n$-dimensional torus. Any non-singular 
$n\times n$ matrix $M$ with integer entries 
defines a continuous surjective endomorphism $f_M$ of $\T^n$ given by 
$$
f_M(\x + \Z^n) \df M\x + \Z^n\quad\forall\,\x\in\R^n\,,
$$
and any continuous surjective endomoprhism $f$ of $\T^n$ can be obtained this way. Criteria for ergodicity of
$f$ (with respect to Haar measure on $\T^n$) are well known, and ergodicity implies that $f$-orbits of 
almost all points are dense in $\T^n$. Also in many cases it is known that exceptional sets of points with non-dense orbits
are rather big. For example, following the notation used in \cite{K},  let us define
\eq{defefy}{
E(f,y) \df \left\{x\in\T^n : y\notin \overline{\{f^k(x): k\in\N\}}  \right\}
}
for a fixed $y\in\T^n$ and a self-map $f$ of $\T^n$. 
In 1988  
Dani proved

\begin{thm}
\label{Daniold} \cite[Theorem 2.1]{D}
For  any semisimple $M\in \GL_n(\R) \cap \M_{n\times n}(\Z)$ and any $y\in\Q^n/\Z^n$,
the set
$ E(f_M,y)$ is $\frac12$-winning.
\end{thm}


The  above winning property 
is based on a game, introduced by 
Schmidt 
in \cite{S1}, which is usually referred to as Schmidt's game. This property
implies density and  full \hd\ and is  stable with respect to countable intersections;
see \S \ref{Schmidt} for more detail. 

One of the goals of the present paper is to prove a far-reaching generalization of Theorem \ref{Daniold}.
Namely, we remove the assumptions of $M$ being semisimple and $y$ being rational. Also
we are able to intersect sets $E(f,y)$ with many `sufficiently regular' fractal subsets of $\T^n$. 
In fact it will be more convenient to lift the problem to $\R^n$: denote by $\pi$ the quotient map
$\R^n\to \T^n$ and, for  $M\in  \M_{n\times n}(\R)$ and $y\in\T^n$,
 consider
\eq{defeay}{
\tilde E(M,y) \df \left\{\x\in\R^n : y\notin \overline{\{\pi(M^k\x): k\in\N\}}  \right\}\,.
}
Clearly $\tilde E(M,y) = \pi^{-1}\big( E(f_M,y)\big)$ when $M\in \GL_n(\R) \cap \M_{n\times n}(\Z)$; 
however the definition  \equ{defeay} makes sense even when $M$ is singular or has non-integer entries. 

The `sufficient regularity' of   subsets of $\rn$ will be characterized by their ability to support 
so-called
{\sl absolutely decaying\/}
measures; 
see \cite{KLW} or \S \ref{measures} for a definition.
\comdima{I agree, let's move the definition to the section about measures}
Examples include $\R^n$ itself and limit sets of irreducible families of 
contracting similarities of $\R^n$ satisfying the Open Set Condition, such as 
the Koch snowflake or the Sierpinski carpet.
Other interesting examples can be found in \cite{KLW, SU, U2}. 

 It turns out, as was first observed 
 in \cite{F}, that 
the absolute decay property 
of a measure
can be used for playing Schmidt's game 
on 
its support.
Namely, we will 
say, following \cite{BBFKW}, that a subset $S$ of $\R^n$ is
{\sl $\alpha$-winning on\/} a subset $K$ of $\R^n$ if $S\cap K$ is $ \alpha$-winning for Schmidt's
game played on the metric space $K$ with the metric induced from $\R^n$. 
From \cite{S1} it immediately follows that the intersection of countably many sets $\alpha$-winning on $K$ is 
also $\alpha$-winning on $K$.
We will say that $S$ is {\sl winning on $K$\/} if it is {$\alpha$-winning on $K$} for some 
$\alpha > 0$.  Precise definitions are given in  \S \ref{Schmidt}.
As a trivial consequence of Corollary \ref{minusone}, if $S$ is winning on 
$K =  \supp\,\mu$, where $\mu$ is absolutely decaying, then $S\cap K$ is 
not
contained in a countable union of affine hyperplanes.
Furthermore, under some additional assumptions on $\mu$,  for example when $K = \R^n$
 or one of the self-similar sets mentioned above, one can show that the \hd\ 
of $S\cap \,K$ is equal to $\dim(K)$ whenever  $S$  is winning on $K$. 
See  \S\ref{measures} for precise statements.

In this paper we prove a generalization of Theorem \ref{Daniold}:

\begin{thm}
\label{Daninew}
For every 
 $K \subset\R^n$ which  supports an  absolutely decaying measure 
 there exists 
 $\alpha = \alpha(K) > 0$ 
such that 
for any $M\in \GL_n(\R) \cap \M_{n\times n}(\Z)$ and  any $y\in\T^n$, 
the set
$\tilde E(M,y)$ is $\alpha$-winning on $K$.
\end{thm}

In particular, for any {\it countable\/} subset $Y$ of $\T^n$, 
the set
$$
\bigcap_{y\in Y}\bigcap_{M\in \GL_n(\R) \cap \M_{n\times n}(\Z)}\tilde E(M,y)
$$
is also $\alpha$-winning on $K$.
It immediately
follows that sets   $ E(f_M,y)$ discussed in  Theorem \ref{Daniold} and their countable intersections
always intersect 
those subsets of the torus 
whose pullbacks to $\R^n$ support absolutely decaying measures. It can also be shown that $\alpha(\R^n) = 1/2$,
recovering Dani's result, see \S\ref{rn}. 

The one-dimensional case  of  Theorem \ref{Daninew}
appeared recently
  in \cite{BBFKW}, and also,
independently and  for $K = \R$, in \cite{Fae}; see also \cite{T1}.
 In other words, the sets
\eq{defeby}{
\tilde E(b,y) \df \left\{x\in\R : y\notin \overline{\{\pi(b^k x): k\in\N\}}  \right\}\,.
}
were shown to be winning
on $\supp\,\mu$ for any absolutely decaying measure $\mu$ on $ \R$,
 any  integer $b > 1$ 
and any $y\in\T$. 
 However, the main result of \cite{BBFKW} applies to much more general situations, recovering
earlier work \cite{Ma, P1, P2} 
by 
Pollington and 
de Mathan. 
In particular, $b$ in \equ{defeby} does not have to be an integer, and   one can replace the sequence
of powers of 
$b$ by an arbitrary lacunary sequence $t_k$ of real numbers
(we recall that $(t_k)$ is called {\sl lacunary\/}  if
$\inf_{k\in \N}\frac{t_{k+1}}{t_k} > 1$.) 

\smallskip
We now describe an analogous generalization of Theorem \ref{Daninew}, which is the main result of 
the present paper. We are going to fix $m,n\in\N$,  consider a sequence
$\cm = (M_k)$ of  $m \times n$ matrices  and a sequence  $\ca = (Z_k)$ of subsets of $\R^m$, and define
\eq{defema}{
\tilde E(\cm,\ca) \df \{\x\in\R^n : \inf_{k\in\N} d(M_k\x ,Z_k) > 0 
 \}\,.
}
(Here 
$d(\cdot,\cdot)$ stands for  the Euclidean distance
on $\R^n$.)
The sets $\tilde E(M,y)$ defined in \equ{defeay} constitute a special case, 
with $m = n$, $\cm = (M^k)$ and $Z_k =  \pi^{-1}(y)$. 

Some assumptions on $\cm$ and $\ca$ are in order. 
We will say that a sequence $\cm$ of nonzero 
$m\times n$ matrices is {\sl lacunary\/}
if so is the sequence $ (\|M_k\|_{op})$ of the values of their operator norms. A subset $Z$ of $\rn$ will be called 
\textsl{$\delta$-uniformly discrete\/} if 
$\inf_{\x ,\y\in Z,\,x\ne y}d(\x ,\y)>\delta$.  With some abuse of terminology, we 
say that a sequence $\ca= (Z_k)$ is $\delta$-uniformly discrete if $Z_k$ is $\delta$-uniformly discrete for every $k\in\N$, and that $\ca$ is  \textsl{uniformly discrete} if it is $\delta$-uniformly discrete for some $\delta > 0$. For example, for an arbitrary sequence $(y_k)$ of points of $\T^m$, 
the sequence of sets $Z_k = \pi^{-1}(y_k)\subset \R^m$ is $1$-uniformly discrete.

\ignore{
We will show that with certain restrictions on $\cm$ and $\ca$, these sets are
winning on supports of absolutely friendly measures, generalizing
several known results (e.g. \cite{D} and \cite{BBFKW}).
It is clear that for the set $\tilde E(\cm,\ca)$ to be large, 
the points in $Z_j$ need to be positively separated; 
for example, if $Z_j$ is dense for some $j$ then clearly $\tilde E(\cm,\ca) = \O$.
Recall that a subset $Z\subset\rn$ is said to be \textsl{uniformly discrete} if there exists $\delta>0$
such that for any two distinct points $\x ,\y\in Z$, one has $d(\x ,\y)>\delta$. 
We will say that such a set is $\delta$-uniformly discrete.
If $Z_j$ is $\delta_j$-uniformly discrete for every $j\in\N$ and $\inf\delta_j = \delta>0$,
we call the sequence $\ca =(Z_j)$
$\delta$-uniformly discrete. 
We call $\ca$ uniformly discrete if it is $\delta$-uniformly discrete for some $\delta > 0$.
We note that trivially any sequence of translates
of $\zn$ is uniformly discrete. This property turns out to be a sufficient restriction
on $\ca$ for our purposes.

For matrices, the crucial property is that their
maximal expansion of vectors grows very quickly in $j$.
More precisely, recall that a real sequence of reals $(\lambda_j)$ is called \textsl{lacunary} if 
$\inf_{j\in \N}\frac{\lambda^{j+1}}{\lambda^j} > 1$,
and call a sequence of matrices lacunary \comlior{values?}
if their operator norms form a lacunary sequence.\\\

Using the above notation}
\medskip

We can now formulate our main result, which is proved in  \S\ref{proofs}:

\begin{thm}
\label{main theorem}
For every 
$K \subset\R^n$ which supports an absolutely decaying measure
\comdima{Not sure whether to have a footnote here or next to the lemma in \S 3, but the impact of \cite{ET} should be mentioned somehow.. what do you think?} 
there exists a positive $\alpha = \alpha(K)$ 
such that if
$\ca$ is a uniformly discrete sequence of subsets of $\R^m$ 
and $\cm$ is a lacunary sequence of $m\times n$ matrices with real entries, then 
$\tilde E(\cm,\ca)$ is $\alpha$-winning on $K$.
\end{thm}


An important  special case  is $m = n$ and $\cm = (M^k)$, where $M$ is an $n\times n$ matrix  with spectral radius strictly greater than $1$ (not necessarily invertible and not necessarily with integer entries);
this is used  to derive Theorem \ref{Daninew} from Theorem \ref{main theorem}, see  \S\ref{proofs}.
Our main theorem also generalizes results from  \cite{BHKV} and  \cite{M} 
dealing with a special case  where \eq{1dimcase}{
\begin{aligned}
\cm\text{ is a lacunary sequence of }1\times n\text{ integer matrices}\\\text{ and }\ca = (Z_k), \text{ where }Z_k = \Z = \pi^{-1}(0)\ \forall\,k\in\N\,.\ \ 
\\
\end{aligned}}
It was observed both in \cite{BHKV} and in \cite{M} that the latter 
set-up can be used to prove the abundance of badly approximable systems of affine forms.
Recall that a pair $(A,\x)$, 
interpreted as a function $\q\mapsto A\q - \x$, $\R^m\to\R^n$ (here $A\in \M_{n\times m}(\R)$ and 
$\x\in\R^n$) is said to be {\sl \ba\/} if  
$$
\inf_{\q\in\Z^m\nz}\|\q\|^{m/n}d(A\q - \x,\zn) > 0\,. $$
This is an inhomogeneous analog of the notion of \ba\ systems of linear forms, see \cite{S2, S3}. 
It was proved in \cite{K2}
 that the set $\mathbf{Bad}(n,m)$ of \ba\ pairs $(A,\x)$ has full \hd. Then a much easier
proof 
was found in \cite{BHKV}, where, for fixed $A\in \M_{n\times m}(\R)$, the sets
$$
\mathbf{Bad}_A(n,m) = \big\{ \x \in \rn :(A,\x)\in \mathbf{Bad}(n,m)\big\}
$$
were considered, and it was shown that $\dim\big(\mathbf{Bad}_A(n,m)\big) = n$ for any $A$.
The latter result was strengthened by Tseng in the case $m = n = 1$: he proved
\cite{T2} that $\mathbf{Bad}_a(1,1)\subset \R$ is 
$\frac18$-winning for any $a\in \R$.
 Shortly thereafter,
Moshchevitin concluded \cite{M} that the sets $\mathbf{Bad}_A(n,m)$ are $\frac12$-winning 
for any $m,n$ and any $A\in \M_{n\times m}(\R)$.
Our main theorem can be used to deduce

\begin{cor}
\label{affine forms}
Le  $K \subset\R^n$ be  the support of an absolutely decaying measure, and let $\alpha$ be as in Theorem \ref{main theorem}.  Then  for any 
$A \in \M_{n\times m}(\R)$, $\mathbf{Bad}_A(n,m)$ is $\alpha$-winning on $K$.
\end{cor}

Independently, in a recent preprint 
\cite{ET} Einsiedler and Tseng provided another proof of this result, 
with a \comdima{I am not sure whose approach is more dynamical :)} 
smaller value of $\alpha$.
We derive Corollary \ref{affine forms} in  \S\ref{proofs}.
At the end of the paper a remark is made
 explaining how all  
 our results can be strengthened 
 to replace `winning' with `strong winning', a property introduced recently in \cite{Fae, FPS, Mc}.

\smallskip

{\bf Acknowledgements:}
The authors are grateful to Manfred Einsiedler, Nikolay Moshchevitin  and Barak Weiss for helpful discussions, and to the referee for useful comments. 
D.K.\  was supported in part by  NSF
Grant DMS-0801064.

\section{Schmidt's game}
\label{Schmidt}
 In this section we describe
the game, first 
introduced by  Schmidt in \cite{S1}.
Let $(X,d)$ be a complete metric space.
Consider  $\Omega \df X \times \mathbb{R}_+$, and define a 
partial ordering
\begin{center}
$(x_2,\rho_2)\le_{s}(x_1,\rho_1)$\  if \  $\rho_2+d(x_1,x_2)\le \rho_1$.
\end{center} 
We associate to each pair $(x,\rho)$ a ball in $(X,d)$ via
$$B(x,\rho) = \{x'\in X: d(x,x') \le \rho\}\,.$$ 
Note that $(x_2,\rho_2)\le_{s}(x_1,\rho_1)$ 
implies (but is not necessarily implied by)
$B(x_2,\rho_2) \subset B(x_1,\rho_1)$. However the two conditions are equivalent when
$X$ is a Euclidean space. 


Schmidt's game is played by two players, whom we  will call
Alice  and Bob, following a convention
 used previously in 
\cite{KW2, BBFKW}. 
The two players are
equipped with parameters $\alpha$ and $\beta $ 
respectively, satisfying $0<\alpha ,\beta <1$. 
 Choose a subset ${S}$ of $ X$ (a target set).
The game starts with Bob picking $x_1\in X$ and $\rho > 0$, 
hence specifying a pair $\omega_1 = (x_1,\rho)$. Alice 
and Bob then take turns choosing $\omega'_k = (x'_k,\rho'_k)\le_s\omega_k$
and $\omega_{k+1}= (x_{k+1},\rho_{k+1})\le_s\omega'_k$ respectively satisfying 
\eq{balls}{\rho_k' = \alpha \rho_k\text{ and }\rho_{k+1} = \beta \rho_k'\,.}
As the game is played on a complete metric space
and the diameters of the nested  balls 
\begin{center}
$B(\omega_1) \supset  B(\omega_1') \supset  \ldots\supset 
B(\omega_k) \supset B(\omega'_k) \supset \ldots$
\end{center}
tend to zero as $k\rightarrow\infty$, 
the intersection of these balls 
is a point $x_\infty\in X$. Call Alice the winner if $x_\infty\in {S}$. 
Otherwise Bob is declared the winner. 
A strategy consists of specifications for a player's choices 
of centers for his or her balls 
given the opponent's previous moves. 

If for certain $\alpha$, $\beta$ and a target set ${S}$ 
Alice has a winning strategy, i.e., 
a strategy for winning the game regardless of how well Bob plays,
we say that ${S}$ is an 
{\sl $(\alpha , \beta)$-winning set\/}.
If ${S}$ and $\alpha$ are such that ${S}$ is an $(\alpha , \beta)$-winning set 
for all possible $\beta$'s, we say that ${S}$ is an 
{\sl $\alpha $-winning\/} set. 
Call a set {\sl winning\/} if such an $\alpha $ exists.

Intuitively one expects winning sets to be large. Indeed, every such set is clearly dense in $X$;
moreover, under some additional assumptions on the metric space winning sets can be proved to
have positive, and even full, \hd.
For example, the fact that a winning subset of $\br^n$ has \hd\ $n$
is due to Schmidt  \cite[Corollary 2]{S1}. Another useful result of Schmidt \cite[Theorem 2]{S1} states that
the intersection of countably many $\alpha$-winning sets 
is $\alpha$-winning. 

Schmidt himself used the machinery of the game he invented to prove that certain
subsets of $\br$ or $\br^n$ are winning, and hence have full \hd. 
Now let $K$ be a closed subset of $X$.
Following an approach initially introduced in \cite{F}, 
we will say that a subset $S$ of $X$ is
{\sl $(\alpha , \beta)$-winning on $K$\/} (resp., {\sl $\alpha $-winning on $K$\/}, {\sl winning on $K$\/})
if $S\cap K$ is $(\alpha , \beta)$-winning  (resp., $\alpha $-winning,  winning) for Schmidt's
game played on the metric space $K$ with the metric induced from $(X,d)$. 
In the present paper we let $X=\R^n$ and take $K$ to be the support of an absolutely decaying measure. 
In other words, since the metric is induced, playing the game on $K$ 
amounts to choosing balls in $\R^n$
according to the rules of a game played on $\R^n$, but with an additional constraint that the
centers of all the balls lie in $K$. Since the first appearance of this approach in \cite{F}, where it was used to
show that sufficiently regular fractals meet with a countable intersection of non-singular affine images of the set of badly approximable vectors in $\R^n$, it has been utilized in  \cite{F2, Kr}, and most recently in \cite{BBFKW}, of which the present paper is a sequel and a generalization.

\section{Absolutely decaying 
measures}
\label{measures}

In this section 
we describe in detail the class of absolutely decaying measures and discuss other related properties 
and their applications.
\comdima{Shortened the intro since the measures have already appeared before.}
Following a terminology introduced in \cite{KLW, PV}, say that a locally finite
Borel measure $\mu$ on $\rn$ is
{\sl $(C,\gamma)$-absolutely decaying\/} if there exists $\rho_{0} >0$ 
such that 
\eq{ad}{\begin{aligned}
\mu\big(B(\x,\rho)\cap \mathcal{L}^{(\varepsilon)}\big) 
	< C(\varepsilon/\rho)^{\gamma}\mu\big(B(\x,\rho)\big)\ \\
	 \text{ for any  affine hyperplane }\mathcal{L}\subset \rn\quad \ \\  \text{
and any }\x\in \supp \, \mu,\  0 < \rho<\rho_0,
	\  \varepsilon > 0\,.
\end{aligned}}

Here $B(\x,\rho)$ stands for the closed Euclidean ball in $\R^n$ of radius $\rho$ centered at $\x$, 
and $\mathcal{L}^{(\varepsilon)} \df \{\x \in\rn : d(\x,\mathcal{L}) \leq \vre\}$ is the closed $\vre$-neighborhood of $\mathcal{L}$. 
We say that  $\mu$ is  {\sl absolutely decaying\/}
if it is $(C,\gamma)$-absolutely decaying for some $C,\gamma > 0$.  (This terminology differs slightly from the
one introduced in \cite{KLW}, where a less uniform version was considered.)
 If $\mu$ is  $(C,\gamma)$-absolutely decaying, we will denote by $\rho_{C,\gamma}(\mu)$ 
 the supremum of $\rho_0$ for which 
\equ{ad} holds.
\ignore{\eq{ad}{\begin{aligned}
\mu\big(B(\x,\rho)&\cap \mathcal{L}^{(\varepsilon\rho)}\big) 
	< C\varepsilon^{\gamma}\mu\big(B(\x,\rho)\big) \quad \text{ for any }\,\text{affine hyperplane }\mathcal{L}\subset \rn \\ & \text{
and any }\x\in \supp \, \mu,\  0 < \rho<\rho_0,
	\ 0 < \varepsilon < 1\,.
\end{aligned}}
(Here $\mathcal{L}^{(t)} \df \{\x \in\rn : d(\x,\mathcal{L}) \leq t\}$ is the closed $t$-neighborhood of $\mathcal{L}$.)
\end{defn}

\begin{defn}
\label{federer}
We say that $\mu$ is
{\sl $D$-Federer\/} if there exists $\rho_{0} >0$ 
such that
\eq{fed}{\mu\big(B(\x ,2\rho)\big) < D\mu\big(B(\x,\rho)\big)\,\quad\forall\,\x\in \supp \, \mu,\ \forall\,0 < \rho<\rho_0\,.}
\end{defn}}

Another property, which often comes in a package
with absolute decay, is the so-called doubling, or Federer, condition. One says that  $\mu$ is
{\sl $D$-Federer\/} if there exists $\rho_{0} >0$ 
such that
\eq{fed}{\mu\big(B(\x ,2\rho)\big) < D\mu\big(B(\x,\rho)\big)\,\quad\forall\,\x\in \supp \, \mu,\ \forall\,0 < \rho<\rho_0\,,}
and {\sl Federer\/} if it is $D$-Federer for some $D>0$. 
Measures which are both absolutely decaying and Federer are called {\sl absolutely friendly\/}, 
a term coined in \cite{PV}.
\ignore{A measure $\mu$ is called {\sl $(C,\gamma, D)$-absolutely friendly} if it is
both $(C,\gamma)$-absolutely decaying and $D$-Federer. 
We will say that $\mu$ is {\sl absolutely decaying\/} (resp.\  {\sl Federer\/}, {\sl absolutely friendly})
if it is $(C,\gamma)$-absolutely decaying (resp.\ $D$-Federer, 
$(C, \gamma, D)$-absolutely friendly)
for some values of constants $C,\gamma,D$. 
If $\mu$ is $(C,\gamma)$-absolutely decaying
(resp., $(C,\gamma, D)$-absolutely friendly), we will denote by $\rho_{C,\gamma}(\mu)$ 
(resp., $\rho_{C,\gamma, D}(\mu)$) the supremum of $\rho_0$ for which 
\equ{ad} holds (resp., both \equ{ad} and \equ{fed} hold).}

Many examples of absolutely friendly measures can be found in \cite{KLW, KW1, U2, SU}. The Federer
condition is very well studied; it obviously holds when  $\mu$  satisfies a {\sl power law\/}, i.e.\  there exist positive $\delta, c_1,c_2,\rho_0$
such that 
\eq{pl}{c_1\rho ^{\delta}\leq\mu\big(B(\x,\rho)\big)\leq c_2\rho^{\delta}\,\quad\forall\,\x\in \supp \, \mu,\ \forall\,0 < \rho<\rho_0 \,.} Such measures are often referred to as {\sl $\delta$-Ahlfors regular\/}.
However it is not hard to construct absolutely friendly measures not satisfying a power law, see \cite{KW1}
for an example. Also, when $n = 1$ the Federer property is implied by the absolute decay, which in its turn is 
implied by a power law (see \cite{BBFKW} for a thorough discussion of  equivalent definitions of absolute friendliness in the one-dimensional case). However these implications fail to hold in higher dimensions.
In particular, the volume measures on smooth $k$-dimensional submanifolds of $\R^n$ 
obviously 
are $k$-Ahlfors regular but   not absolutely decaying unless $k = n$. 

\ignore{
Clearly the $n$-dimensional Lebesgue measure is absolutely friendly as well as any 
Borel measure $\mu$ on $\rn$ satisfying a power law, i.e.\  there exist positive $\gamma, k_1,k_2,\rho_0$
such that for every $\x\in \supp \, \mu$ and $0 < \rho<\rho_0$ one has
$$k_1\rho ^{\gamma}\leq\mu\big(B(\x,\rho)\big)\leq k_2\rho^{\gamma}\,.\\
$$}

The goal of the current work, as well as in several earlier papers \cite{KW1, KTV, F, F2, Kr}, is to use
measures in order to construct points in their supports with prescribed (dynamical or Diophantine) properties.
Our attention will therefore be focused on closed subsets $K$ of $\rn$ which support absolutely decaying and absolutely friendly measures.
\ignore{, and
it will be convenient to introduce the following  terminology:

\begin{defn}
\label{sets}
Let $K$ be a  closed subset of  $\rn$ and let $C, \gamma, D > 0$.
Say that $K$ is
{\sl $(C,\gamma)$-absolutely decaying\/} (resp., {\sl absolutely decaying\/})  if there exists a $(C,\gamma)$-absolutely decaying (resp., {absolutely decaying}) measure $\mu$ on $\R^n$ with $K = \supp\,\mu$. Similarly define {\sl $(C,\gamma, D)$-absolutely friendly\/} and {\sl absolutely friendly\/} sets.
\end{defn}}
For example, this is the case when $K = \R^n$, or when $K$ is the limit set of an irreducible family of contracting self-similar \cite{KLW} or self-conformal \cite{U2} transformations of $\R^n$ satisfying the Open Set Condition. More examples can be found in \cite{KW1, SU}. Note that the paper \cite{BHKV}  established full \hd\ of $\tilde E(\cm,\ca)\cap K$ for $\cm,\ca$ as in 
\equ{1dimcase} and under an assumption that $K\subset \R^n$ supports an absolutely decaying, $\delta$-Ahlfors regular
measure with $\delta > n-1$. It is not hard to show, using an elementary covering argument, that 
\equ{pl} with $\delta > n-1$ implies \equ{ad} with $\gamma = \delta - n+1$. 
Hence the sets considered 
in  \cite{BHKV} support absolutely decaying measures.

\smallskip

Recall that the {\sl lower pointwise dimension\/} of a measure  $\mu$ at $\x\in\supp\,\mu$ is defined as
$$\underline{d}_\mu(\x) \df \liminf_{\rho\to 0} \frac{\log\mu(B(\x,\rho))}{\log \rho}.
$$
For an open $U$ with $\mu(U) > 0$ let
\eq{dmu}{\underline{d}_\mu(U) \df \inf_{\x\in \supp\,\mu\cap U}\ \underline{d}_\mu(\x) \,.}
It is well known, see e.g.\ \cite[Proposition 4.9]{Fa}, that \equ{dmu} constitutes a lower bound for the \hd\ of $\supp\,\mu\cap U$ (where this bound is sharp when 
$\mu$ satisfies a power law). It is also easy to see that
$\underline{d}_\mu(\x)\ge \gamma$ for every $\x\in \supp \, \mu$ 
whenever $\mu$  is $(C,\gamma)$-absolutely decaying:  indeed, 
take $\rho < \rho_0 < \rho_{C,\gamma}(\mu)$ and $\x\in \supp \, \mu$; then, 
using \equ{ad} and noting that $B(\x,\rho)\subset\cl^{(\rho)}$ for some hyperplane $\cl$, one has
$$\mu\big(B(\x,\rho)\big) < C\left(\frac{\rho}{\rho_0}\right)^\gamma\mu\big ( B(\x,\rho_0)\big).$$
Thus, for $\rho < 1$,
$$
\frac{\log \mu\big(B(\x,\rho)\big)}{\log \rho} \ge \gamma 
+ \frac{\log C - \gamma \log \rho_0 + \log \mu\big(B(\x,\rho_0)\big)}{\log \rho},
$$
and the claim follows.

The following proposition 
 \cite[Proposition 5.1]{KW2} makes it possible to estimate
the \hd\ of sets winning on supports of Federer measures: 

\begin{prop} \label {dimension}Let $K$ be the support of a Federer measure $\mu$  on 
$\R^n$, 
and let $S$ be winning on $K$. Then for any open $U\subset \R^n$ 
with $\mu(U) > 0$ one has 
$$\dim(S\cap K\cap U) \ge
 \underline{d}_\mu(U)\,.$$
\end{prop}

In particular, if in addition $\mu$ is $(C,\gamma)$-absolutely decaying,  in the above proposition one can replace $\underline{d}_\mu(U)$ 
with $\gamma$, and with $\dim(K)$ if $\mu$
satisfies a power law.
Note that this generalizes  estimates for the \hd\ of winning sets due to Schmidt \cite{S1}
 for $\mu$ being Lebesgue  on
$\R^n$, and to Fishman  \cite[\S 5]{F} for measures satisfying a power law.
\medskip

The next lemma exhibits a crucial feature of  sets supporting 
absolutely decaying measures,
namely the fact that while playing Schmidt's game on such a set,
Alice can distance herself from hyperplanes `efficiently'.
This observation is the cornerstone of the proof of our main theorem.
The argument has been adapted from the one in \cite{M}, where the case $K = \rn$ was proved
with $\alpha = \frac12$ (see \S\ref{rn} for more detail), and then refined using an observation from \cite{ET}.
\comdima{Maybe we don't have to say it again, or do we?}

\begin{lem}
\label{log turns}
For every $C, \gamma>0$ and
\eq{alpha}{
\alpha < \frac{1}{2C^{1/\gamma}+1}
}
there exists 
$\varepsilon=\varepsilon(C,\gamma,\alpha)\in (0,1)$ such that if $K$ is the support of a 
$(C,\gamma)$-absolutely decaying  measure $\mu$ on $\rn$,
$0 < \rho < \rho_{C,\gamma}(\mu)$, $\x_1 \in K$, $N\in \N$, and $\cl_1,\dots,\cl_N$ are hyperplanes in $\rn$, 
there exists $\x_2\in K$ with
\eq{containment}{B(\x_2,\alpha\rho) \subset B(\x_1,\rho)}
and 
\eq{distance}{d(B(\x_2,\alpha\rho),\cl_i) > \alpha\rho\quad\text{for at least $\lceil\varepsilon N\rceil$ of the hyperplanes }\cl_i.}
\end{lem}

\begin{proof}
Let $A_i = B\big(\x_1,(1-\alpha)\rho\big)\ssm \cl_i^{(2\alpha\rho)}$.
By 
\equ{ad} and \equ{alpha}, for each $1 \leq i \leq N$,
$$\frac{\mu(A_i)}{\mu\big(B(\x_1,(1-\alpha)\rho)\big)} > 1 - C\left(\frac{2\alpha}{1-\alpha}\right)^\gamma\df \varepsilon > 0.$$
We claim there exist $j_1, \dots, j_k$, where $k \geq \lceil\varepsilon N\rceil$, such that
$K \cap \bigcap_{i=1}^k A_{j_i} \neq \varnothing$.
To see this, let $f(\x) = \sum_{i=1}^N \chi_{A_i}(\x)$.
Then
$$\displaystyle\int_{B(\x_1,(1-\alpha)\rho)} f(\x)\, d\mu(\x) \geq N \varepsilon\mu\big(B(\x_1,(1-\alpha)\rho)\big),$$
so clearly there exists some $\x_2 \in K$ with $f(\x_2) \geq N\epsilon$. Since
$f(\x_2) \in \Z$, there must exist $j_1,\dots,j_k$ as above.
Hence, $\x_2$ satisfies \equ{containment} and \equ{distance}.
\end{proof}

We will also need the following corollary of the above lemma:

\begin{cor}\label{minusone}
Let $K$ be the support of
a $(C,\gamma)$-absolutely decaying measure on 
$\rn$, 
 let $\alpha$ be as in \equ{alpha},  let $S\subset \rn$ be $\alpha$-winning on $K$,
 and let $S'\subset S$ be a countable union of hyperplanes. Then $S\ssm S'$ is also  $\alpha$-winning on $K$.
\end{cor}

\begin{proof} In view of the countable intersection property,
it suffices to show that for any hyperplane $\cl\subset\rn$, the set $\rn\ssm\cl$ is $(\alpha,\beta)$-winning on $K$ for any $\beta$. Let $\mu$ be a
$(C,\gamma)$-absolutely decaying measure with $K = \supp\,\mu$.
We let Alice play arbitrarily until the radius of a ball chosen by Bob is
less than $\rho_{C,\gamma}(\mu)$. Then apply Lemma \ref{log turns} with $N = 1$ and $\cl_1 = \cl$,
which yields a ball disjoint from $\cl$. Afterwards she can keep playing arbitrarily, winning the game.
\end{proof}

\section{Proofs}\label{proofs}

Let us now state a more precise version of Theorem \ref{main theorem}:

\begin{thm}
\label{precise theorem}
Let $K$ be the support of
a $(C,\gamma)$-absolutely decaying measure on 
$\rn$, and
 let $\alpha$ be as in \equ{alpha}. Then for any
uniformly discrete sequence $\ca$ of subsets of $\R^m$ 
and any  lacunary sequence $\cm$ of $m\times n$ real matrices, the set
$\tilde E(\cm,\ca)$ is $\alpha$-winning on $K$.
\end{thm}

\begin{proof}
Write $\cm = (M_k)$, let
$t_k \df \|M_k\|_{{op}}$
and let $\vbf_k$ be a unit vector satisfying
$$ \|M_k\vbf_k\| =t_k.$$
Take $\delta > 0$ such that $\ca$ is $\delta$-uniformly discrete, and let
 \eq{ratio}{\inf_k \frac{t_{k+1}}{t_k} = Q > 1\,.} Now pick an arbitrary $0 < \beta < 1$, 
 take  $\varepsilon$ as in Lemma \ref{log turns}, and 
choose $N$ large enough that 
\eq{large N}
{(\alpha\beta)^{-r} \leq Q^N\text{, where } r =  \lfloor \log_{\frac{1}{1-\varepsilon}} N\rfloor +1\,.}
We will denote
by $M_k^{-1}(Z)$ the preimage of a set $Z \subset \rn$ under $M_k$.
Notice that for each $k\in\N$, $M_k^{-1}(Z_k)$ is contained in
a countable union of \ignore{positively separated} hyperplanes,
so applying Corollary \ref{minusone} a finite number of times, we may assume that $t_1 \geq 1$.

By playing arbitrary moves if needed, 
we may assume without loss of generality that $B(\omega_1)$ has radius
\eq{rho small}{\rho
	< \min\left(\frac{\alpha\beta\delta}{4},\rho_{C,\gamma}\right)\,.} 
Now let \eq{defc}{c = \min\left(\rho (\alpha\beta)^{2r-1},\frac{\delta}{4}\right).}
We will describe a strategy for Alice to play the $(\alpha,\beta)$-game on $K$ and 
to ensure that for all $j\in\N$,  for all 
$\x \in B(\omega_{r(j+1)}')$ and for all $k$ with  $1 \leq t_k < 
(\alpha \beta)^{-r j}$, one has
$d(M_k\x,Z_k) > c $. This will imply that 
$\bigcap_k B(\omega^{\prime}_k) \in \tilde E(\cm,\ca) \cap K$, finishing the proof.

To satisfy the above goal, Alice can choose $\omega'_i$ arbitrarily for $i < r$.
Now fix $j\in \N$. By \equ{ratio} and \equ{large N}, 
there are at most $N$ indices $k \in \N$ for which 
\eq{indices}{(\alpha\beta)^{-r (j-1)} \leq t_k  < (\alpha \beta)^{-r j}.}
Let $k$ be one of these indices.
For any $\x \in\rn$,
$ \| \x \| 
	\geq \frac{1}{t_k}\|M_k(\x)\|$.
Thus, if $\y_1,\y_2$ are two different points in  $Z_k$, then by \equ{rho small} and \equ{indices}
\eq{dist2}{d\Big(M_k^{-1}\big(B(\y_1,c)\big),M_k^{-1}\big(B(\y_2,c)\big)\Big) 
	\geq \frac{\delta-2c}{t_k} \geq \frac{\delta}{2t_k}
		> \frac{\delta}{2}(\alpha\beta)^{rj} \geq 2\rho(\alpha\beta)^{rj-1};}
therefore $B(\omega_{rj})$ intersects with at most one set of the form $M_k^{-1}\big(B(\y,c)\big)$, where
$\y\in Z_k$. Hence, for each $k$ satisfying \equ{indices},
\eq{Z preimage}
{B(\omega_{rj}) 
\cap M_k^{-1}(Z_k^{(c)}) 
\subset M_k^{-1}\big(B(\y,c)\big) \text{ for some }\y\in Z_k.}
We will now show that the preimage of such a ball is contained in a `small enough' neighborhood of some hyperplane, so that we can apply the decay condition. Toward this end,
let $V\subset\rmb$ be the hyperplane perpendicular to $M_k\vbf_k$ and passing through
${\mathbf{0}}$. Then $$W \df M_k^{-1}(V)$$ is a hyperplane in $\rn$ passing through ${\mathbf{0}}$.

\ignore{
To see this, first note that $W$ is a subspace, by linearity of $M_j$. Let $\vbf_j, \ubf_1,\dots, \ubf_{m-1}$
be a basis for $\rmb$ such that $\ubf_1 \notin W$ . Then $M_j\ubf_1 = \w + tM_j\vbf_j$,
where $\w \in W$ and $t \in \R$. But then $\vbf_j, \ubf_1-t\vbf_j,\ubf_2,\ubf_3,\dots,\ubf_{m-1}$
is also a basis for $\rmb$ and $\ubf_1-t\vbf_j \in W$. By induction, we can find a basis
$\vbf_j,\ubf_1^{'}, \dots, \ubf_{m-1}^{'}$ for $\rmb$ with $\ubf_i^{'} \in W$ for $i=1,\dots,m-1$.
Clearly, $\vbf_j \not\in W$, so $W \neq \rmb$. Hence, 
$W$ must be a hyperplane passing through $\mathbf{0}$.}

If $\x\not\in W^{(c/t_k)}$, then $\x = \w + \eta \vbf_k$ for some
$\eta > c/t_k$ and $\w\in W$, thus
$$\|M_k\x\| = \|M_k\w + M_k\eta\vbf_k\| \geq \eta\|M_k\vbf_k\| = t_k\eta > c\,.$$
Hence, $M_k^{-1}\big(B({\mathbf 0},c)\big) \subset W^{(c/t_k)}$, which clearly implies that for each $\y\in Z_k$,
$M_k^{-1}\big(B(\y,c)\big) \subset \cl^{(c/t_k)}$ for some hyperplane $\cl \subset \rn$.
By \equ{defc} and \equ{indices},
$$\frac{c}{t_k} \leq (\alpha\beta)^{r(j+1)-1}\rho\df \zeta\,.$$
Therefore, by \equ{Z preimage},
\eq{thickness}{\displaystyle\bigcup_{t_k \text{ satisfies \equ{indices} }} B(\omega_{rj})\cap M_k^{-1}\big(Z_k^{(c)}\big)	\subset \bigcup_{i=1}^N \cl_i^{(\zeta)},}
where $\cl_i$ are hyperplanes.
Noticing that by \equ{large N} $(1-\varepsilon)^rN <1$, Alice can utilize
Lemma \ref{log turns} $r$ times to distance herself by $\zeta$ from each of the
hyperplanes $\cl_i$ after $r$ turns.
Thus for $k$ satisfying \equ{indices},
it holds that 
$$B(\omega_{r(j+1)}') \cap M_k^{-1}\big(Z_k^{(c)}\big) = \varnothing\,.$$
We conclude that $d(M_k\x, Z_k) \geq c$ for any $\x \in B(\omega_{r(j+1)}')$, which implies the desired statement.
\end{proof}


\ignore{The following simple observation will be useful.

\begin{cor}
\label{finun}
If $\ct$ is a finite union of lacunary sequences, then $\tilde E(\ct,\cy)$ is $\alpha$-winning,
where $\alpha$ is as in Lemma \ref{log turns}.
\end{cor}
}

\ignore{
\begin{remark}
Clearly, taking the matrices above to be in $\text{Mat}_{n\times n}(\Z)$, we have
that for a large class of sequences $\cm$ of toral endomorphisms, $\tilde E(\cm,\cy)$
is winning. In particular, one can take $\cm = (M^j)$, where $M$ is semi-simple. 
(see \cite[Theorem ?]{D})
\end{remark}
}

\begin{proof}[Proof of Theorem \ref{Daninew}] Recall that we are given $M \in \GL_n(\R) \cap \M_{n}(\Z)$.
If all the eigenvalues of $M$
have modulus less than or equal to $1$, then obviously
every eigenvalue of $M$ must have modulus $1$.
By a theorem of Kronecker \cite{Kro}, they must be roots of unity,
so there exists an $N\in\N$ such that the only eigenvalue of $M^N$ is $1$. 
Let $J = L^{-1}M^NL$ be the Jordan normal form of $M^N$,
and let $\vbf_i = L\ve_i$, $i = 1,\dots,n$, be the Jordan basis for $M^N$. 
Then, 
since $M^N$ is an integer matrix, 
we have 
$\vbf_i \in \Q^n$ for each $1\leq i \leq n$.
Hence, letting $V = \text{span}(\vbf_1,\dots,\vbf_{n-1})$,
$V+\Z^n$ is a union of positively separated parallel hyperplanes.
Since $J$
fixes the last coordinate
of any vector, if $a_1,\dots,a_n \in \R$, then
$$M^N\left(\displaystyle\sum_{i=1}^n a_i\vbf_i\right) 
	\in a_n\vbf_n + V.$$
Therefore, for $\x, \y\in \R^n$ with $\x-\y \not\in V+\Z^n$ and any $k\in \N$ 
one has $$d(M^{Nk}\x,\y+\Z^n) \geq c_0 d(\x-\y,V+\Z^n) > 0,$$
where $c_0$ is a positive constant depending only on $\vbf_1,\dots,\vbf_n$.
Hence, for any $y \in\T^n$, $$\tilde E(M^N,y) \supset  \R^n \smallsetminus (\pi^{-1}(y) + V) = \R^n \smallsetminus (\y+V+\Z^n)\,,$$
where $\y$ is an arbitrary vector in $\pi^{-1}(y)$. Thus $\tilde E(M^N,y)$
is $\alpha$-winning on $K$ by Corollary \ref{minusone}.
Hence $\tilde E(M^N, z)$ is $\alpha$-winning on $K$ whenever $z\in f_M^{-i}(y)$, where $0 \leq i < N$.
Thus the intersection $$\tilde E(M,y) = \bigcap_{i=0}^{N-1}\bigcap_{z\in f_M^{-i}(y)} \tilde E\big(M^N, f_M^{-i}(y)\big)$$
is also $\alpha$-winning on $K$.
\smallskip

In the case where at least one of the eigenvalues is of absolute value strictly greater than 1,
we will show that the sequence $(\| M^k \|_{op})$ is a finite union of lacunary sequences, 
which will clearly imply that $\tilde E\big((M^k),\ca\big)$ is $\alpha$-winning on $K$.
Let $J = L^{-1}ML$ be the Jordan normal form of $M$. Since the operator norm of
$M$ as a real transformation is equal to its operator norm as a complex transformation
and $$\|J^k\|_{op} \leq \|L\|_{op}\|L^{-1}\|_{op}\|M^k\|_{op} \,\,\,\,\,
	\text{and}\,\,\,\,\,\,\|M^k\|_{op} \leq \|L\|_{op}\|L^{-1}\|_{op}\|J^k\|_{op},$$
letting $c = \|L\|_{op}\|L^{-1}\|_{op}$, we have
$\frac{1}{c}\| M^k\|_{op} \leq \| J^k\|_{op} \leq c\|M^k\|_{op}$ for all $k\in\N$.
\ignore{
Hence, for all $m,n\in\N$,
$$
\frac{\|M^{m}\|_{op}}{\|R^{n}\|_{op}} \geq \frac{c_1}{c_2}\frac{\|J^{m}\|_{op}}{\|J^{n}\|_{op}}.
$$
}
Hence, if $(\|J^k\|_{op})$ is eventually lacunary, then there exists $\ell, N \in \N$ 
and $Q > 1$ such that,
for all $k\geq N$,
$$
\frac{\|M^{k+\ell}\|_{op}}{\|M^{k}\|_{op}} \geq \frac{1}{c^2}\frac{\|J^{k+\ell}\|_{op}}{\|J^{k}\|_{op}} 
	\geq Q\,.
$$
Thus  it will suffice to show that $(\|J^k\|_{op})$ is eventually lacunary.

Let $B$ be an $m\times m$ block of $J$ associated to an eigenvalue $\lambda$
and write $B^k = \big(b_{ij}{(k)}\big)$.
Direct computation shows that, for $0 \leq j-i \leq k$,
\eq{entries}{
b_{ij}{(k)} = {k\choose j-i}\lambda^{k-(j-i)},
}
and $b_{ij}{(k)} = 0$ otherwise.
Since $|b_{ij}(k)| = o(|b_{1m}(k)|)$ as functions of $k$ for all $(i,j) \neq (1,m)$,
\eq{domentry}{
\lim_{k\to\infty} \frac{\|B^k\|_{op}}{|b_{1m}(k)|} = 1.
}
Hence,
\eq{block}{
\lim_{k\to\infty} \frac{\|B^{k+1}\|_{op}}{\|B^{k}\|_{op}} = |\lambda|,
}
so clearly if $|\lambda| >1$ then $(\|B^k\|_{op})$ is eventually lacunary. 
Write $J= B_1 \oplus \dots \oplus B_s$,
where $s\in\N$ and $B_i$ are the Jordan blocks, with associated eigenvalues $\lambda_i$.
Let $\lambda_{max}= \max |\lambda_i|$,
and let $B_{max}$ be a block with associated eigenvalue 
having absolute value $\lambda_{max}$ and of maximal dimension among
such blocks.
By \equ{entries} and \equ{domentry}, for any $i$,
$$\lim_{k\to\infty} \frac{\|B_{max}^k \|_{op}}{\|B_i^k \|_{op}} \geq 1.$$
Hence, by \equ{block},
$$
\lim_{k\to\infty} \frac{\|J^{k+1} \|_{op}}{\|J^k \|_{op}}
	= \lim_{k\to\infty} \frac{\|B_{max}^{k+1} \|_{op}}{\|B_{max}^k \|_{op}} = \lambda_{max}.
$$
Since by assumption $M$ (and therefore $J$) has an eigenvalue with absolute value greater than $1$,
$(\|J^k\|_{op})$ is eventually lacunary.
\end{proof}

\ignore{Let $\lambda$ be the largest absolute value of any eigenvalue of $R$.
Let $\vbf$ be any real, normed vector contained in the space spanned by eigenvectors
having eigenvalue with absolute value $\lambda$.
Since $R$ is semisimple, there exist $c_1,c_2> 0$ such that

$$\|R^j\vbf\| \geq c_1\lambda^j\|\vbf\|$$
and, for all $\ubf \in \rn$,
$$\|R^j\ubf \| \leq c_2\lambda^j\|\ubf\|.$$
Let $k\in\N$ be large enough that $\lambda^k > \frac{c_2}{c_1}$.
Then, for any normed $\ubf \in \rn$,
$$\frac{\|R^{j+k}\vbf\|}{\|R^j\ubf \|} \geq \frac{c_1\lambda^{j+k}}{c_2\lambda^j} 
	= \frac{c_1}{c_2}\lambda^k > 1$$
Hence, $\car_i = (R^{i+jk})_{j\in\N}$ is a lacunary sequence of matrices for each $i = 1,2,\dots,k$,
so letting $\cy_i = (y_{i+jk})_{j\in\N}$, $\tilde E(\car_i,\cy_i)$ is $\alpha$-winning,
where $\alpha$ is as in Theorem \ref{main theorem}.
Since $\tilde E(R,\cy)= \cap_{i=1}^k \tilde E(\car_i,\cy_i)$, the corollary follows.}

\ignore{
\begin{cor}
For any $C, \gamma, D > 0$ and $n\in\N$, there exists $\alpha = \alpha(C,\gamma,D,n) > 0$
such that if $K$ is the support of a $(C,\gamma,D)$-absolutely friendly measure on $\rn$,
$A \subset \Tn$ is countable, then $\cap \tilde E(R,A)$ is $\alpha$-winning,
where $R$ ranges over all semisimple surjective endomorphisms of $\Tn$.
\end{cor}
}

In the remaining part of this section we apply Theorem \ref{main theorem} to badly approximable systems of affine forms.

\begin{proof}[Proof of 
 Corollary \ref{affine forms}] Recall that we need to fix $A\in \M_{n\times m}(\R)$ and study the set
 $$
\mathbf{Bad}_A(n,m) = \left\{ \x \in \rn : \inf_{\q\in\Z^m\nz}\|\q\|^{m/n}d(A\q - \x,\zn) > 0\right\}\,.
$$
First observe  that the above set is easy to understand in the `rational' case
when there exists a nonzero $\ubf \in \Z^n$
such that $A^T\ubf \in \Z^m$ (or equivalently, when the rank of the group $A^T\Z^n + \Z^m$ 
is strictly smaller than $ m+n$). In this case, by a theorem of Kronecker, see  \cite[Ch.\ III, Theorem IV]{C}, 
$ \inf_{\q\in\Z^m}d(A\mathbf{q}-\x,\Z^n) $ is positive
 if and only if the value of $\ubf\cdot \x$ is not an integer.
Therefore 
$$\mathbf{Bad}_A(n,m)  \supset \{\x\in \R^n : \ubf\cdot \x \notin \Z\}.$$
Since the right-hand side is the complement of a countable union of hyperplanes, 
in view of  Corollary \ref{minusone}
$\mathbf{Bad}_A(n,m)$ is $\alpha$-winning on $K$ whenever $K$ is absolutely decaying 
 and $\alpha$  is as in Theorem \ref{main theorem}.

\ignore{
\begin{thm}[Kronecker's Theorem]

For $A \in \M_{n\times m}(\R)$ and $\x\in\R^n$, the following are equivalent:
\begin{enumerate}
\item There exists $c > 0$ such that for any $\mathbf{q}\in \Z^m$, $d(A\mathbf{q}-\x,\Z^n) > c$.
\item There exists $\ubf\in \Z^n$ such that $A^T\ubf \in \Z^m$ but $\x\cdot \ubf \notin \Z$.
\end{enumerate}
\end{thm}

Notice that for $m=n=1$ the above conditions are clearly equivalent:
If $A = \frac{k}{\ell}$ then $d(Aq +x,\Z) \geq d(x,\frac{1}{\ell}\Z)$, which is positive
if and only if $x\notin \frac{1}{\ell}\Z$, precisely when (2) holds; if $A\notin \Q$, then for all $x\in\R$,
$(A\q + x)\mod 1$ is dense in $[0,1]$, so neither condition holds.

So when $m=n=1$ and $A = \frac{k}{\ell}$, 
$A^T\ubf \in \Z^m$ whenever $\ubf \in \ell \Z$, so the theorem implies that the complement of 
$\mathbf{Bad}_A(1,1)$ is contained in the countable set $\frac{1}{\ell}\Z$.

First consider the case that there is some $\ubf \in \Z^n\setminus\{\mathbf{0}\}$
such that $A^T\ubf \in \Z^m$. Then, by Kronecker's Theorem, if $\x\cdot \ubf \not\in \Z$,
then, for some $c > 0$, $d(A\mathbf{q} - \x,\Z^n) > c$ for all $\mathbf{q}\in \Z^m$.
Hence, $$\mathbf{Bad}_A(n,m) \supset \{\x\in \R^n : \x\cdot \ubf \notin \Z\}.$$
Since the right-hand side is the complement of a countable union of hyperplanes,
$\mathbf{Bad}_A(n,m)$ is $\alpha$-winning by Corollary \ref{minusone}.}

\smallskip
In the more interesting `irrational' case 
when $\text{rank}(A^T\Z^n + \Z^m) = m+n$, 
one can utilize the theory of best approximations to $A$ as developed by Cassels
\cite[Ch.\ III]{C} and recently made more precise by Bugeaud and Laurent \cite{BL}.\ignore{To see this, suppose $\text{rank}(A^T\Z^n + \Z^m) < m+n$.
Then $\Z^m+\sum_{i=1}^{j} \Z A^T\mathbf{e}_i = \Z^m+\sum_{i=1}^{j+1} \Z A^T\mathbf{e}_i$
for some $1\leq j \leq n-1$, so $A^T\mathbf{e}_{j+1} = \mathbf{z} + \sum_{i=1}^jA^T(a_i\mathbf{e}_i )$ for some $a_i \in \Z$. 
Taking $\ubf = \mathbf{e}_{j+1}-\sum_{i=1}^j a_i\mathbf{e}_i$, we have $A^T\ubf \in \Z^m$
and $\ubf \neq 0$.}
In \cite[\S\S 5--6]{BHKV}, using results from \cite{BL}, it is shown that if $\text{rank}(A^T\Z^n + \Z^m) = m+n$, then
there exists a lacunary sequence of vectors $\y_k \in \Z^n$ (a subsequence of the sequence of
best approximations to $A$) such that whenever $\x\in\rn$ satisfies $$\inf_{k\in\N}d(\y_k\cdot \x, \Z) > 0\,,$$
it follows that $\x\in \mathbf{Bad}_A(n,m)$. In other words, 
 $$\tilde E(\cy,\ca) \subset \mathbf{Bad}_A(n,m)\,,$$
 where $\cy \df (\y_k)$ and $\ca = (Z_k)$ with $Z_k = \Z$ for each $k$. (See also \cite[\S2]{M} for an alternative exposition.)
Therefore in this case $\mathbf{Bad}_A(n,m)$ is $\alpha$-winning on $K$ by Theorem \ref{main theorem}.
\end{proof}

\section{Concluding remarks}\name{finalremarks}
\subsection{Playing on $\rn$ with $\alpha = 1/2$}\name{rn} As was mentioned before, the special case
$K = \R^n$ of our main theorem is essentially contained in \cite{M}. In fact, arguing as in \S \ref{proofs} and
using \cite[Lemma 2]{M} (the analogue of our Lemma \ref{log turns}) and  \cite[Lemma 3]{M} (Schmidt's escaping lemma, cf.\ \cite[Ch.\ 3, Lemma 1B] {S3}), one can show that for  $\ca$  and  
 $\cm$  as in Theorem \ref{main theorem} and any  $\alpha,\beta > 0$ with $1 + \alpha\beta - 2\alpha> 0$,  the sets
$\tilde E(\cm,\ca)$ are $(\alpha,\beta)$-winning. In particular, this shows that one can take $\alpha(\R^n) = 1/2$
in Theorems \ref{Daninew} and  \ref{main theorem}.

\subsection{Strong winning}\name{strong} Recently in \cite{Fae, FPS} and independently
in \cite{Mc} a
 modification of Schmidt's game has been introduced, where
condition \equ{balls} is replaced by \eq{strongballs}{\rho_k' \geq \alpha \rho_k\text{ and }\rho_{k+1} \geq \beta \rho_k'\,.}
Following  \cite{Mc}, a subset $S$ of a metric space $ X$
is said to be {\sl $(\alpha,\beta)$-strong winning\/} if Alice
 has a winning strategy in the game defined by  \equ{strongballs}.
Analogously, one defines $\alpha$-strong winning and strong winning sets.  
It is not hard to verify that
strong winning implies winning (see \cite{FPS} for a proof), 
and that a countable intersection of $\alpha$-strong winning sets is $\alpha$-strong winning. Furthermore, this class has stronger invariance properties,
e.g.\ it is proved in \cite{Mc} that strong winning subsets of $\R^n$ are preserved by quasisymmetric 
homeomorphisms.

It is not hard to modify the proofs given above to show that in Theorem \ref{main theorem}
(and therefore in all its corollaries), $\alpha$-winning may be replaced by
$\alpha$-strong winning. This is done by adding `dummy moves' in order to
accommodate the possibly slower decrease in radii of the chosen balls. Details will appear elsewhere.

\bibliographystyle{alpha}

\end{document}